\documentclass[12pt]{amsart}

\pdfoutput=1

\usepackage{amsmath}
\usepackage{amsthm}
\usepackage[foot]{amsaddr}
\usepackage{amsfonts}
\usepackage{amssymb}
\usepackage{graphics}
\usepackage[mathscr]{euscript}
\usepackage{listings}
\usepackage{enumitem}
\usepackage{xcolor}
\usepackage{setspace}
\usepackage[colorlinks=true, linkcolor=red, citecolor=blue]{hyperref}

\newcommand{\bb}[1]{\ensuremath{\mathbb{#1}}}
\newcommand{\bbm}[1]{\ensuremath{\mathbf{#1}}}
\newcommand{\ca}[1]{\ensuremath{\mathcal{#1}}}
\newcommand{\N}{\ensuremath{\mathbb{N}}}
\newcommand{\defeq}{\mathrel{\overset{\makebox[0pt]{\mbox{\normalfont\tiny\sffamily
					def}}}{=}}}

\newcommand{\E}[1]{\ensuremath{\mathbf{E}\left[#1\right]}}
\newcommand{\Var}[1]{\ensuremath{\mathbf{Var}\left[#1\right]}}
\newcommand{\Cov}[1]{\ensuremath{\mathbf{Cov}\left[#1\right]}}
\newcommand{\Bin}{\ensuremath{\text{Bin}}}

\newcommand{\adj}{\sim}
\newcommand{\given}{\ensuremath{\middle|}}

\renewcommand{\o}{\ensuremath{\text{o}}}

\newcommand{\I}{{\mathbf{I}}}
\renewcommand{\P}{{\mathbf{P}}}
\newcommand{\J}{{\mathbf{J}}}

\newcommand{\eps}{{\varepsilon}}

\newcommand{\abs}[1]{\ensuremath{\left| #1 \right|}}
\newcommand{\card}[1]{\ensuremath{\big| #1 \big|}}

\renewcommand{\Pr}{{\mathbf P}}
\DeclareMathOperator*{\Pro}{\mathbf{P}}
\newcommand{\univto}{\xrightarrow{\text{univ}}}

\renewcommand{\iff}{\Leftrightarrow}

\newcommand{\dif}{\operatorname{dif}}
\renewcommand{\epsilon}{\eps}

\theoremstyle{plain}
\newtheorem{theorem}{Theorem}
\newtheorem{conjecture}[theorem]{Conjecture}

\newtheorem{lemma}[theorem]{Lemma}
\newtheorem{corollary}[theorem]{Corollary}

\theoremstyle{definition}
\newtheorem{definition}[theorem]{Definition}

\lstdefinestyle{myStyle}{
language=Python,
basicstyle=\footnotesize\tt,
keywordstyle=\color{orange!70!black},
commentstyle=\itshape\color{gray!60!black},
stringstyle=\color{green!50!black},
stepnumber=1,
tabsize=2,
numbers=none,
numberstyle=\tiny,
numbersep=5pt,
showspaces=false,
escapechar=`,
showstringspaces=false
}
\include{psfig}

\calclayout

\title[The power of few]{Reaching a consensus on random networks: \linebreak the power of few}

\date{\today}

\author[Linh Tran]{Linh Tran}
\email{l.tran@yale.edu}
\address{Dept. of Mathematics, Yale University, 10 Hillhouse Ave, CT 06511}
\author[Van Vu]{Van Vu}
\email{van.vu@yale.edu}

\allowdisplaybreaks

\begin{document}
	
\begin{abstract}
A community of $n$ individuals splits into two camps, Red and Blue. The individuals are connected by a social network, which influences their colors. Everyday, each person changes his/her color according to the majority among his/her neighbors. Red (Blue) wins if everyone in the community becomes Red (Blue) at some point. 

We study this process when the underlying network is the random Erdos-Renyi graph $G(n, p)$. With a balanced initial state ($n/2$ person in each camp), it is clear that each color wins with the same probability. 

Our study reveals that for any  constants $p$ and $\varepsilon$, there is a constant $C$ such that if one camp has  $n/2 +C$  individuals, then it wins with probability at least $1 - \varepsilon$. The surprising key fact here is that $C$ {\it does not} depend on $n$, the population of the community. When $p=1/2$ and $\varepsilon =.1$, one can set  $C$ as small as 6.  If the aim of the process is to choose a candidate, then this means it takes only $6$ ``defectors"  to win an election unanimously with overwhelming odd.
\end{abstract}


%
%

\maketitle


\newpage

\section {Introduction} \label{sec:introduction}

\subsection {The opinion exchange  dynamics} 

Building mathematical models to explain how collective opinions are formed is an important and interesting task (see \cite{mossel2017} for a survey on the topic, with examples from  various fields, economy, sociology, statistical physics, to mention a few). 

Obviously, our opinions are influenced by people around us, and this motivates the study of the following natural and simple model. 
  A community of $n$ individuals splits into two camps,  Red and Blue, representing two competing opinions, which can be on any topic such as brand competition, politics, ethical issues, etc. The individuals are connected by a social network, which influences  their opinion on a daily basis (by some specific rule). We say that Red (respectively Blue) \emph{wins} if everyone in the community becomes Red (respectively Blue) at some point.

We study  this process  when the underlying network is random. In this paper, we focus on the Erdos-Renyi random graph $G(n, p)$, which is the most popular model of random graphs  \cite{Bollobasbook, Jansonbook}. We use majority rule, which is also one of the most natural rules. When a new day comes, a vertex will take the color which is dominant among his neighbors. If there is a tie, it keeps its color.

\begin{definition} The random graph $G(n,p)$ on $n$ vertices is obtained by putting an edge between any two vertices with probability $p$, independently. 
\end{definition}

\subsection{Results}
With a balanced initial state ($n/2$ persons in each camp), by symmetry, each color wins with the same  probability $q < 1/2$, regardless of $p$.   (Notice that there are graphs, such as the empty and complete graphs, on which no one wins.)

Our study reveals  that for any given $p$ and $\eps$, there is  a constant $c$ such that if one camp has $\frac{n}{2} + c$ individuals at the initial state, then 
it wins with probability at least $1- \eps$.  The surprising fact here is that $c$ does not depend on $n$, the population of the community. When $p=1/2$ and $\eps =.1$, one can set 
$c$ as small as 6.

\begin{theorem} [The power of few]  \label{thm:const-advantage-ex}
	Consider the  (majority) process on $G(n, 1/2)$. Assume that the Red camp has at least  $\frac{n}{2} + 6$  vertices at the initial state, where $n \ge 550$. Then  Red wins after the fourth day with probability at least  $93\%$. 
\end{theorem}

This result can be stated without the Erd\"os-Renyi model; one can state an equivalent theorem by choosing a graph, uniformly, from the set of all graphs on $n$ vertices,
to be the network. 

This  result reveals an interesting phenomenon, which we call ``the power of few". 
The collective outcome can be extremely sensitive, as a  modification of the smallest scale  in the initial setting leads to 
the opposite outcome.  

Our result applies in the following equivalent settings.

{\it Model 1.} We fix the two camps of size $n/2 +6$ and $n/2-6 $, respectively,  and draw a random graph on their union.

{\it Models 2.} Draw a random graph first, let Red be a random subset of $n/2 +6$ vertices (chosen uniformly from all subsets of that size), and Blue be the rest. 

{\it Model 3.}  Split the society into two camps of size $n/2$ each. Draw the random graph on their union, then recolor 6 random selected Blue vertices to Red.

{\it Model 4.} Split the society into two camps (Red and Blue) of size $n/2-6$ each and a ``swinging" group (with no color yet) of 12 individuals. Draw the random graph on their union. Now let the swinging group join the Red camp. 

With Model 3, we can imagine a balanced election process at the beginning. Then 6 persons change camp. This tiny group already guarantees the final win with an overwhelming odds. Similarly, Model 4 asserts that a swinging group of size 12 decides the outcome. 

Our result can also be used to model the phenomenon that outcomes  in seemingly identical 
situations can be   the opposites. Consider two communities, each has exactly $n$ individuals and shares  the same social network. In the first community, Red camp  has 
size $n/2 +6$, and Blue camp has $n/2-6$. In the second community, Blue camp has $n/2+6$ and Red camp has $n/2-6$. If $n$ is large, then 
there is no way to tell the difference between the two communities. Even if we record everyone's initial opinion,  clerical errors will surely swallow the tiny difference
of 12.  However, at the end, the collective opinion  will be opposite, with high probability.

Now we state the general result for arbitrary constant density $p$. 

\begin{theorem}[Asymptotic bound] \label{thm:unbounded-advantage}
Let $p$ be a constant in $(0, 1)$ and $c_n$ be a positive integer which may depend on $n$
.
Assume that Red has $n/2 +c_n$ individuals in day zero and the random graph is $G(n,p)$. 
Then  Red wins after the fourth day with probability at least $1 - K(p) \max \{n^{-1}, c_n^{-2} \} $, where $K(p)$ depends only on $p$. 
\end{theorem}

Both  results follow from Theorem \ref{thm:const-advantage},  which, in a slightly technical form, describes how the process evolves day by day. 
Our results can be extended to cover the case when there are more than 2 opinions; details will appear in a later paper \cite{TranVusparse}.

\subsection{Related results and Sparse graphs} 

The majority dynamics on random graphs has been studied recently by many researchers \cite{mossel2017, mossel2014, fountoulakis, benjamini}.  The key difference, as compared to our study,  is in the set-ups. In these earlier papers, each individual chooses his/her initial color uniformly at random. The central limit theorem thus guarantees that with high probability, the initial difference between the two camps is of order $\Theta (\sqrt n) $. Therefore, these papers did not touch upon the ``power of few" phenomenon, which is our key message. On the other hand,  they considered sparse random graphs where the density $p$ goes to zero with $n$. 

In \cite{benjamini}, Benjamini, Chan, O'Donnell, Tamuz, and Tan considered random graphs with $p \ge \lambda n^{-1/2} $, where $\lambda $ is a sufficiently large constant, and showed that the dominating color wins with probability at least $.4$ \cite[Theorem 1.2]{benjamini}, while conjecturing that this probability in fact tends to $1$ as $n \to \infty$. This conjecture was proved by Fountoulakis, Kang, and Makai
\cite[Theorem 1.1]{fountoulakis}.

\begin{theorem}  \label{FKM} 
For any $0< \eps \le 1$ there is $\lambda = \lambda(\eps)$ such that the following holds for $p \ge \lambda n^{-1/2} $. 
 With probability at least $1 -\eps$, over the choice of the random graph $G(n,p)$ 
	and the choice of the initial state,  the dominating color wins after four days. 
\end{theorem} 

For related results on random regular graphs, see \cite{mossel2014, mossel2017}. 
Building upon ideas developed in this paper,  we can extend our main theorem to accommodate sparse graphs as follows.

\begin{theorem} \label{thm:general-sparse}
For any $0< \eps \le 1$ there is $c = c (\eps)$ such that the following holds for $p \ge (2+\o(1)) (\log n)/n$. 
Assume that Red camp has size at least $n/2 + c/p$ initially, then it wins with probability at least $1 -\eps$. 
\end{theorem}

The proof, along with additional information, including the length of the process and the explicit relation between $\eps$ and $c$, will appear in a subsequent paper \cite{TranVusparse}. 
Notice that when $p$ is a constant, this result covers the ``Power of Few" phenomenon 
 as a special case, albeit with potentially larger $c$. The result no longer holds for $p < (\log n)/n$ as in this case there are, with high probability, isolated vertices. Any of these vertices  keeps it original color forever. In this case, the number of Blue vertices  converges, and we obtain a bound on the limit in \cite{TranVusparse}.

One can use  Theorem \ref{thm:general-sparse} to derive a ``delayed" version  (in which Red may need more than 4 days) of Theorem \ref{FKM}, as follows.

\begin{proof}[Proof of Theorem \ref{FKM}]
Assume Theorem \ref{thm:general-sparse}. Let $R_0$ and $B_0$ respectively be the initial Red and Blue camps.
Fix  a constant $0 < c' \le \eps/6$. $\card{R_0}\sim \Bin(n, 1/2)$ since it is a sum of $\Bin(1, 1/2)$ variables. An application of the Berry-Esseen theorem (Corollary \ref{lem:strongesseen}; with $C_0 =.56$) implies that 
\begin{equation*}
\Pr \left(\card{R_0} - \frac{n}{2} \le c'\sqrt{n} \right) \le \Phi(2c') + \frac{C_0}{\sqrt{n}},
\end{equation*}
and 
\begin{equation*} 
\Pr \left(\card{R_0} - \frac{n}{2} \le -c'\sqrt{n} \right) \ge \Phi(-2c') - \frac{C_0}{\sqrt{n}} , 
\end{equation*} 
Thus
\begin{equation*}
\begin{aligned}
& \Pr \left(\abs{\card{R_0} - \frac{n}{2}} \le c\sqrt{n} \right)
\ \le \ \left(\Phi(2c') + \frac{C_0}{\sqrt{n}}\right)
- \left(\Phi(-2c') - \frac{C_0}{\sqrt{n}}\right) \\
& \ \le \ \Phi(-2c', 2c') + \frac{2C_0}{\sqrt{n}}
\ \le \  \frac{4c'}{\sqrt{2\pi}} + \frac{2C_0}{\sqrt{n}}
\ \le \ \frac{\eps}{3} + \frac{2C_0}{\sqrt{n}} \le \epsilon/2, 
\end{aligned}
\end{equation*} for sufficiently large $n$.

On the other hand, if $\abs{\card{R_0} - n/2} > c'\sqrt{n}$, then one of the sides has more than $n/2 + c' \sqrt{n}$ initial members, which we call the \emph{majority side}.
Now we apply Theorem \ref{thm:general-sparse} with $\epsilon$ replaced by $\epsilon/2$. 
Notice that in the setting of Theorem \ref{FKM} if  we have $p =\lambda n^{-1/2} $ for $\lambda$ sufficiently large, then $c' \sqrt n \ge c/ p$, where $c$ is the constant in Theorem \ref{thm:general-sparse}. Thus, by  this theorem,  the probability for the majority side to win is at least  $1-\eps/2$, and we are done by the union bound.

\end{proof}

\subsection {Notation}

\begin{itemize}
	\item $R_t, B_t$: Respectively the sets of Red and Blue vertices after day $t$. (At this point each person has updated his color $t$ times.)
	\item $\I_t(u) \ \defeq \ \bbm{1}_{\{u\in R_t\}}$: $\{0, 1\}$-indicator of the event that $u$ is Red after day $t$.
	\item $\J_t(u) \ \defeq \ 2\I_t(u) - 1$: $\{-1, 1\}$-indicator of the same event.
	\item $u\adj v \equiv (u, v)\in E$: Event that $u$ and $v$ are adjacent.
	\item $\Gamma(v) \ \defeq \ \{u: u\adj v\}$: The neighborhood of $v$.
	\item $W_{uv} \ \defeq \ \mathbf{1}_{\{u\adj v\}}$ - Indicator of the adjacency between $u$ and $v$.
	
	\item $\ca{N}(\mu, \sigma^2)$: The Normal Distribution with mean $\mu$ and variance $\sigma^2$.
	\item $\Phi(a, b) \ \defeq \ (2\pi)^{-1/2}\int_a^be^{-\frac{x^2}{2}}dx\ $ and $\ \Phi(a) \defeq \Phi(-\infty, a), \quad \Phi_0(a) \defeq \Phi(0, a)$.
\end{itemize}

\subsection{Main Theorem} \label{subsec:main-theorem}

%
%

The main theorem concerns dense graphs, where $p$ is at least a constant.  When given appropriate parameter values, it implies the ``Power of Few'' phenomenon in Theorem \ref{thm:const-advantage-ex}.
Before stating the theorem, we define some expressions.
\begin{align*}
C_0 \defeq 0.56, & \quad C_1 \defeq \sqrt{3\log 2}, \quad \sigma = \sigma(p) \defeq \sqrt{p(1 - p)}, \\
P_1 = P_1(n, p, c, \eps_2) & \ \defeq \ \frac{1/4 + 4C_0^2\left(1 - 2\sigma^2\right)^2\cdot \frac{n - 1}{n - 2}} {\left(\sqrt{n - 1} \Phi_0 \left( \frac{2pc + \min\{p, 1- p\}}{\sigma \sqrt{n - 1}} \right) - C_0\frac{1 - 2\sigma^2}{\sigma} - \frac{C_1}{p} - \eps_2 - \frac{1}{2\sqrt{n}}\right)^2}, \\
P_2 = P_2(n, p, \eps_2, \eps_1) & \ \defeq \ \frac{1}{n}\exp\left(-\frac{8n}{3} \left[(1 - 2\eps_1)p\eps_2(p\eps_2 + C_1) - \eps_1\log 2\right] \right), \\
P_3 = P_3(n, p, \eps_1, r) & \ \defeq \ \frac{1}{n}\exp\left(-\frac{2rp^3(2\eps_1 n - 1)^2}{1 + rp} - 2n\log 2\right), \\
P_4 = P_4(n, p, r) & \ \defeq \ n\exp\left(- (1 - 2r + 2r\log 2r)p (n-1)\right).
\end{align*}

\begin{theorem} \label{thm:const-advantage}
	Let $p\in (0, 1)$, $c \in \bb{N}$, $n\in \bb{N}$, $r \in \left(0, \frac{1}{2}\right)$ and $\eps_1, \eps_2 >0$. Let $C_0, C_1$, $\sigma$, $T_1, T_2$, $P_1, P_2, P_3, P_4$ as above. Assume that 
	\begin{equation} \label{eq:n-eps1-eps2-conds}
	2\sqrt{n - 1} \Phi_0 \left( \frac{2pc + \min\{p, 1- p\}}{\sigma \sqrt{n - 1}} \right) > \frac{C_1}{p} + 2\eps_2 + \frac{1}{\sqrt{n}} \quad \text{ and } \quad 2\eps_1n > 1.
	\end{equation}
	With $n^R, n^B$ being integers such that $n^R + n^B = n$ and $1\le n^B \le \frac{n}{2} - c$,  the  process on  $G\sim G(n, p)$ with $|B_0|= n^B$  satisfies the following
	\begin{enumerate}
	\item With  $\ n^B_0 \defeq \frac{n}{2} - c$, $\ n^B_1 \defeq \frac{n - 1}{2} - \left(\frac{C_1}{2p} + \eps_2\right) \sqrt{n}$, $\ n^B_2 = \left(\frac{1}{2} - \eps_1\right)n$, $\ n^B_3 = rp(n - 1)$,  $\ n^B_4 = 0$, 
	\begin{equation*}
	\quad \P \left(\card{B_t} \le n^B_t\ \given \ \card{B_{t-1}}\le n^B_{t-1} \right) \ge 1 - P_t \ \text{ for each } \ t=1,2,3,4.
	\end{equation*}
	\item $ \displaystyle
	\P \left(R_4 = V(G) \ \given \  |B_0| = n^B \right) \ge
	1 - (P_1 + P_2 + P_3 + P_4).
	$
\end{enumerate}
\end{theorem}

Intuitively, $P_i$ is a upper bound on the probability of some abnormal event happening at Day $i$. If none of these events occurs, the whole population 
becomes Red after Day Four.  The proof for this theorem occupies the remaining sections of the paper. The next section (Section 2) contains a few lemmas. We start with analyzing the situation after the first day  in Section \ref{sec:firstday-const}. The situations in subsequent days will be studied in Section \ref{sec:universal-win}. All details are assembled in Section \ref{sec:proof-main} to form the full proof of the theorem.
%

In the rest of this section, we  derive Theorems \ref{thm:const-advantage-ex} and  \ref{thm:unbounded-advantage} from Theorem  \ref{thm:const-advantage}, discuss related results in the literature, and state a few open questions. 

\begin{proof}[Proof of Theorem \ref{thm:const-advantage-ex}]
	Assuming Theorem \ref{thm:const-advantage}. Observe that if the conditions in \eqref{eq:n-eps1-eps2-conds} hold for some value of $n$, then they hold for all larger values of $n$. 
	 Let $n = 550$, $\eps_1 = \eps_2 = 0.01$, $r = 0.3$ (along with $p = 0.5$ and $c = 6$), we have \eqref{eq:n-eps1-eps2-conds} satisfied. Furthermore, a routine calculation shows that 
	\begin{equation*}
	P_1\le 0.06866, \quad P_2\le 0.00144, \quad P_3 \le 0.00010, \quad P_4 \le 0.00005 , \end{equation*} 
\noindent which implies that $ \P \left(B_4\neq \varnothing \right) < 0.07$ or equivalently that Red wins in the fourth day with probability at least $.93$
(conditioned on the event $\card{B_0} = n^B\le \frac{n}{2} - c$). 
\end{proof}

\begin{proof}[Proof of Theorem \ref{thm:unbounded-advantage}]
In this proof, only $n$ and $c =c_n$ can vary. We can assume, without loss of generality, that $c_n \le n/2$.
	Assuming Theorem \ref{thm:const-advantage}, we choose (constants) $\eps_1, \eps_2$ such that
	$(1 - 2\eps_1)p\eps_2(p\eps_2 + C_1) - \eps_1\log 2 > 0,$ and arbitrary $r$,
	then a routine calculation shows that $P_2, P_3, P_4 = \o(n^{-2})$ and $P_1 = \Omega(n^{-1})$, so $P_1 + P_2 + P_3 + P_4 =(1+o(1)) P_1$. We have, for sufficiently large $n$,
	\begin{equation*}
	\begin{array}{c}
	\sqrt{n - 1}\ \Phi_0 \left( \tfrac{2pc_n + \min\{p, 1- p\}}{\sigma \sqrt{n - 1}} \right) - \tfrac{C_0(1 - 2\sigma^2)}{\sigma}- \tfrac{C_1}{p} - \eps_2 - \tfrac{1}{2\sqrt{n}} \ \ge \ \tfrac{\sqrt{n}}{2} \ \Phi_0 \left( \tfrac{2c_n\sqrt{p}}{\sqrt{n}} \right) = \ \tfrac{T(n)}{2}, \\[3pt]
	1/4 + 4C_0^2(1 - 2\sigma^2)^2 \tfrac{n-1}{n-2} \ \le \ 1/4 + 4 \cdot 0.6^2 \cdot 1.5 \ < \ 3.
	\end{array}
	\end{equation*}
	Thus, $P_1 \le 12T(n)^{-2}$. It then suffices to show $T(n)  \ge H(p)\min\{c_n^{-2}, n^{-1}\}$ for some term $H(p)$ depending solely on $p$. Consider  2 cases:
	
\noindent If $c_n \ge \sqrt{n}$, then:
$ \displaystyle 
T(n) \ \ge \ \sqrt{n}\cdot \Phi_0(2\sqrt{p}) \ \ge \ \sqrt{n}\cdot 2\sqrt{p} \cdot \frac{\Phi_0(2)}{2} \ = \ \sqrt{p} \ \Phi_0(2) \sqrt{n}. 
$

\noindent If $c_n < \sqrt{n}$, then:
$ \displaystyle
T(n) \ \ge \ \sqrt{n} \cdot  \frac{2c_n\sqrt{p}}{\sqrt{n}} \cdot \frac{\Phi_0(2)}{2} \ = \ \sqrt{p}\ \Phi_0(2) \ c_n.
$

\noindent In any case, $T(n) \ge H(p)\min\{c_n^{-2}, n^{-1}\}$, for $H(p) = \Phi_0(2)\sqrt{p}$, as desired.
\end{proof}


\subsection {Open questions}

Let $\rho(k,n)$ be the probability that Red win if its camp has size $n/2 +k$ in the beginning, when $p = .5$. Theorem \ref{thm:const-advantage-ex} shows that $\rho (6, n) \ge .93$ (given that $n$ is sufficiently large).
In other words, five defectors guarantee Red's victory with an overwhelming odd. In fact, we have $\rho(4, n) \ge .73$ by plugging in the same values for $\eps_1$, $\eps_2$ and $r$ with $c = 4$ in Theorem \ref{thm:const-advantage-ex}'s proof.
We conjecture that one defector already brings a non-trivial advantage. 

\begin{conjecture} (The power of one)
There is a constant $\delta >0$ such that   $\rho (1,n) \ge 1/2 +\delta $ for all sufficiently large $n$. 
\end{conjecture}

In the following numerical experiment,  we run $T = 10000$ independent trials.  In each trial, we  fix a set of $N = 10000$ nodes with $5001$ Red and $4999$  Blue (meaning $c=1$), generate a graph from $G(N, 1/2)$,  and simulate the process on the resulting graph. We record the number of wins  and the number of days to achieve the win in percentage  in Table \ref{tab:voting-sim}. Among others, we see that Red wins within $3$ days with frequency more than $.93$.
The source code for the simulation along with execution
instructions can be found online at \url{https://github.com/thbl2012/majority-dynamics-simulation}.
\begin{table}[h]

	
	\begin{tabular}{|ccccccrr|}\hline
			T & p & Red & Blue & 	Winner & Last day & Count & Frequency \\
			\hline
			$10^4$ & $1/2$ & $5001$ & $4999$ &	Blue & 3 & 496 & 4.96 \ \% \\
			$10^4$ & $1/2$ & $5001$ & $4999$ &	Blue & 4 & 77 & 0.77 \ \% \\
			$10^4$ & $1/2$ & $5001$ & $4999$ &	Blue & 5 & 3 & 0.03 \ \% \\
			$10^4$ & $1/2$ & $5001$ & $4999$ &	Blue & 7 & 1 & 0.01 \ \% \\
			$10^4$ & $1/2$ & $5001$ & $4999$ &	Red & 2 & 25 & 0.25 \ \% \\
			$10^4$ & $1/2$ & $5001$ & $4999$ &	Red & 3 & 9313 & 93.13 \ \% \\
			$10^4$ & $1/2$ & $5001$ & $4999$ &	Red & 4 & 85 & 0.85 \ \% \\
			\hline
	\end{tabular}
	\vspace{3pt}
	
	\caption{Winners and winning days with their frequencies}
	\label{tab:voting-sim}
\end{table}

Imagine that people defect from Blue camp to Red camp one by one. The {\it value} of the $i$th defector is defined as $v(i,n)= \rho (i,n) - \rho (i-1, n)$ (where we take $\rho (n,0) =1/2 $). It is intuitive to think that the values of the defectors decrease. (Clearly defector number $n/2$ adds no value.)

\begin{conjecture} (Values of defectors)
For any fixed $i$ and sufficiently large $n$, we have $v(i,n) \ge v(i+1, n)$. 
\end{conjecture}

It is clear that the second conjecture implies that first one, with $\delta = \frac{.4}{5} =.08 $, although the simulation results above suggests that $\delta$ can be at least $.43$.


\section {Some Probabilistic Lemmas}

\subsubsection*{Hoeffding's inequality }
Hoeffding's inequality is a classical result that gives exponentially small probability tails for sums of independent random variables.
\begin{theorem}[Hoeffding]  \label{thm:hoeffding}
	Let $\{X_i\}_{i=1}^n$ be independent random variables and $\{a_i\}_{i=1}^n$, $\{b_i\}_{i=1}^n$, such that for all $i=1,2,\ldots, n$, $a_i\le X_i\le b_i$ almost surely. Then for $X = X_1+X_2+\cdots + X_n$, we have
	\begin{equation*}
	\max \Bigl\{\Pr\left( X - \E{X} \ge t \right),\ \Pr\left( X - \E{X} \le -t \right)^{} \Bigr\} \ \le \ \exp\left(-\frac{2t^2}{\sum_{i=1}^n (b_i - a_i)^2}\right).
	\end{equation*}
\end{theorem}
The proof of Hoeffding's inequality is available in most graduate level probability textbooks, e.g. \cite{vershynin2018}. The original proof given by Hoeffding appeared  in \cite{hoeffding1963}.



\subsubsection*{Berry-Esseen type inequalities for Bernoulli variables}
Berry-Esseen theorem in its classical form establishes an explicit asymptotic bound for the convergence rate of sums of random variables to the normal distribution.
\begin{theorem}[Berry-Esseen] \label{thm:esseen}
	There is a universal constant $C_0$ such that for any $n$, if $X_1, X_2, \ldots, X_n$ are random variables with zero mean, variances $\sigma_1^2, \sigma_2^2, \ldots, \sigma_n^2 > 0$, and absolute third moments $\E{|X_i|^3} = \rho_i < \infty$, we have:
	\begin{equation*} \label{BE1} 
	\sup_{x\in \bb{R}}\abs{\Pr\left( \frac{\sum_{i=1}^n X_i}{\sigma_X} \le x \right) - \Phi(x)}
	\ \le \ C_0\sigma_X^{-3/2} \sum_{i=1}^n\rho_i
	\quad \text{ where } \sigma_X^2 = \sum_{i=1}^n\sigma_i^2.
	\end{equation*}
\end{theorem}

The original proof by Esseen \cite{esseen} yielded  $C_0 = 7.59$, and this constant has been improved a number of times. The latest work by Shevtsova \cite{shevtsova} achieved  $C_0 = 0.56$, which will be used for the rest of the paper.
We will be interested in the setting where   $\{X_i\}_{i=1}^n$ are r.v.s taking values in either $\{0, 1\}$ or $\{0, -1\}$. The following corollary and lemmas naturally follow from this theorem.

\begin{corollary} \label{lem:strongesseen}
	Let $p\in (0, 1)$ and $X_1, X_2, \ldots, X_n$ be Bernoulli random variables such that for all $i$, either
	$X_i \sim \Bin(1, p)$ or $-X_i \sim \Bin(1, p)$. Let
	$X = X_1 + X_2 + \cdots + X_n$ and $\mu_X = \E{X}$.
	Then,
	\begin{equation*}  \label{BE3}
	\sup_{x\in \bb{R}}\abs{\Pr\left(X - \mu_X \le x\right) - \Phi\left(\frac{x}{\sigma\sqrt{n}}\right)}
	\ \le \ \frac{C_0(1 - 2\sigma^2)}{\sigma\sqrt{n}} \quad \text{ where } \sigma = \sqrt{p(1-p)}.
	\end{equation*}
\end{corollary}

\subsubsection*{Difference of Binomial Random Variables}

\begin{lemma} \label{thm:diff of binom r.v. gen.}
	For $p\in (0, 1)$, $\sigma = \sqrt{p(1-p)}$ and $n_1, n_2 \in \bb{N}$ such that $n_1 > n_2$.
	let $Y_1 \sim \Bin(n_1, p), \ Y_2 \sim \Bin(n_2, p)$ be independent random variables.
	Then for any integer $d < p(n_1 - n_2)$ 
	\begin{equation*}
	\Pr\left(Y_1 > Y_2 + d\right) \ \ge \ \frac{1}{2} + \Phi_0\left(\frac{p(n_1 - n_2) - d}{\sigma\sqrt{n_1 + n_2}}\right) - \frac{C_0\left(1 - 2\sigma^2\right)}{\sigma\sqrt{n_1 + n_2}}.
	\end{equation*}
\end{lemma}

\begin{lemma} \label{thm:1/sqrt{n}}
	Let $p \in (0, 1)$ be a constant and $\sigma = \sqrt{p(1-p)}$,  $X_1 \sim \Bin(n_1, p)$ and $X_2 \sim \Bin(n_2, p)$ be independent r.v.s. Then for any positive integer $d < \frac{n_1+n_2}{2}$,
	\begin{equation*}
	\Pr\left(X_1 = X_2 + d \right) \ \le \ \frac{2C_0\left(1 - 2\sigma^2\right)}{\sigma\sqrt{n_1 + n_2}}.
	\end{equation*}
\end{lemma}

\section{Day One} \label{sec:firstday-const}

We  first analyze  the situation after Day One. The main result for this section is Theorem \ref{thm:const-advantage-day-1}, which  bounds on the probability that we still have ``many'' Blue vertices left. 
Firstly, let us recall a few terms defined in Section \ref{subsec:main-theorem}.
\begin{equation*}
\sigma \ \defeq \ \sqrt{p(1 - p)},
\quad
P_1 \ \defeq \ \frac{1/4 + 4C_0^2\left(1 - 2\sigma^2\right)^2\cdot \frac{n - 1}{n - 2}} {\left(\sqrt{n - 1} \Phi_0 \left( \frac{2pc + \min\{p, 1- p\}}{\sigma \sqrt{n - 1}} \right) - \frac{C_0(1 - 2\sigma^2)}{\sigma} - \frac{C_1}{2p} - \eps_2 - \frac{1}{2\sqrt{n}}\right)^2}.
\end{equation*}
Define a new term $Q$ by
\begin{equation*}
Q = Q(n, p, d) \ \defeq \ \frac{1/4 + 4C_0^2\left(1 - 2\sigma^2\right)^2\cdot \frac{n - 1}{n - 2}} {\left(\sqrt{n - 1} \Phi_0 \left( \frac{2pc + \min\{p, 1- p\}}{\sigma \sqrt{n - 1}} \right) - \frac{C_0(1 - 2\sigma^2)}{\sigma} - d - \frac{1}{2\sqrt{n}}\right)^2}.
\end{equation*}
Observe that $Q\left(n, p, \frac{C_1}{2p} + \eps_2 \right) = P_1(n, p, \eps_2)$. 

\begin{theorem} \label{thm:const-advantage-day-1}
	Let $p\in (0, 1)$ and $c$ be constants and $\sigma, T_1, T_2, Q$ be defined above. Let $n, n^R, n^B\in \bb{N}$ such that $n^R + n^B = n$, $1\le n^B \le \frac{n}{2} - c$, and $d > 0$ such that
  \begin{equation} \label{eq:d-n-condition-defn}
  	\sqrt{n - 1} \Phi_0 \left( \frac{2pc + \min\{p, 1- p\}}{\sigma \sqrt{n - 1}} \right) - \frac{C_0(1 - 2\sigma^2)}{\sigma} \ > \ d + \frac{1}{2\sqrt{n}}.
  \end{equation}
	Then we have
	\begin{equation*} \label{eq:const-advantage-day-1}
	\Pr \left(\card{B_1} > \frac{n-1}{2} - d\sqrt{n} \quad \middle| \  \card{B_0} =n^B\right) \ \le \ Q\left(n, p, \frac{C_1}{2p} + \eps_2 \right)
	\end{equation*}
\end{theorem}

This theorem states that Red will (with a given probability) increase its advantage from $2c$ to at least $2d\sqrt{n} + 1$ after Day One. 

The proof relies on a Chebyshev concentration bound on the sum $\card{B_1} = \sum_{u\in V} \I_1(u)$. The next step is thus analyzing the indicators $\I_1(u)$ and their dependence.

\subsection{Bounds on Day One Red indicators and their dependence}

\begin{lemma} \label{lem:day1-to-binom}
Let $u$ be any vertex, $S = \{v_1, v_2, \ldots, v_r\}$ be a set of vertices. Then
\begin{equation*}
\Pr \left(u\in R_1 \mid W_{uv_i}, \ i = 1,\ldots, r \right)
= \Pr \left(Y_1 > Y_2 - \sum_{i=1}^r \J_0(v_i)W_{vv_i} - \I_0(u)  \right),
\end{equation*}
where $ \quad \left\{
\begin{array}{ll}
Y_1 & \sim \Bin\left(\card{R_0} - \sum_{i=1}^r \I_0(v_i) - \I_0(u), \ p \right) \\
Y_2 & \sim \Bin\left(\card{B_0} + \sum_{i=1}^r \I_0(v_i) + \I_0(u) - r - 1, \ p \right)
\end{array} \right.$.
\end{lemma}

\begin{proof}
Observe that, given the values of $W_{uv_i}$ for $i = 1,\ldots, r$, we have
\begin{equation*}
u\in R_1 \iff \sum_{v\in V} \J_0(v)W_{uv} + \I_0(v) > 0 
\iff \sum_{v\in (V\setminus S)\setminus \{u\}} \J_0(v)W_{uv} > -\I_0(u) - \sum_{i=1}^r \J_0(v_i)W_{uv_i}.
\end{equation*}
We have $(V\setminus S)\setminus \{u\} = V_1 \cup V_2$ where $V_1 \defeq (R_0\setminus S)\setminus \{u\}$ and $V_2 \defeq (B_0\setminus S)\setminus \{u\}$.
Since $(-1)^{l-1}\sum_{v\in V_l} \J_0(v)W_{uv} \sim \Bin(\card{V_l}, p)$ for $l = 1, 2$ and
\begin{equation*}
\begin{aligned}
\card{V_1} & \ = \ \card{R_0} - \sum_{i=1}^r \I_0(v_i) - \I_0(u), \\
\card{V_2} & \ = \ \card{B_0} - \sum_{i=1}^r \left(1 - \I_0(v_i)\right) - \left(1 - \I_0(u)\right) \ = \ \card{B_0} +  \sum_{i=1}^r \I_0(v_i) + \I_0(u) - r - 1
\end{aligned},
\end{equation*}
the result follows.
\end{proof}

This lemma leads to the following two useful corollaries.

\begin{corollary} \label{col:day1-exp}
Under the same setting as in Lemma \ref{lem:day1-to-binom}, we have
\begin{equation*}
\E { \I_1(u) \mid W_{uv_i}, \ i = 1,\ldots, r }
\ge \frac{1}{2} + \Phi_0\left(
\frac{T}{\sigma \sqrt{n - r - 1}}\right) - \frac{C_0\left(1 - 2\sigma^2\right)}{\sigma\sqrt{n - r - 1}},
\end{equation*}
where $T\ \defeq \ \displaystyle 2pc + \sum_{i=1}^r \J_0(v_i)(W_{uv_i} - p) + \min\{p, 1-p\}$.
\end{corollary}

\begin{proof}
From lemmas \ref{lem:day1-to-binom} and \ref{thm:diff of binom r.v. gen.}, we get
\begin{equation*}
\Pr \left(u\in R_1 \mid W_{uv_i}, \ i = 1,\ldots, r \right)
\ge \frac{1}{2} + \Phi_0\left(\frac{p(n_1 - n_2) - d}{\sigma\sqrt{n_1 + n_2}}\right) - \frac{C_0\left(1 - 2\sigma^2\right)}{\sigma\sqrt{n_1 + n_2}},
\end{equation*}
where $n_1 = \card{R_0} - \sum_{i=1}^t \I_0(v_i) - \I_0(v)$, $n_2 = \card{B_0} +  \sum_{i=1}^t \I_0(v_i) + \I_0(v) - r - 1$ and $d = -\I_0(u) - \sum_{i=1}^r \J_0(v_i)W_{uv_i}$. Therefore
\begin{align*}
p(n_1 - n_2) - d & \ = \ p \left( \card{R_0} - \card{B_0} - \sum_{i=1}^r \J_0(v_i) - \J_0(v) \right) + \I_0(u) + \sum_{i=1}^r \J_0(v_i)W_{uv_i} \\
& \ = \ p(\card{R_0} - \card{B_0}) + \sum_{i=1}^r \J_0(v_i)(W_{uv_i} - p) + (1 - 2p)\I_0(u) + p \\
n_1 + n_2 & \ = \ \card{R_0} + \card{B_0} - r - 1 = n - r - 1.
\end{align*}
Since $\E { \I_1(u) \mid W_{uv_i}, \ i = 1,\ldots, r } = \Pr \left(u\in R_1 \mid W_{uv_i}, \ i = 1,\ldots, r \right)$, $\card{R_0} - \card{B_0} \ge 2c$ and $(1 - 2p)\I_0(u) + p \ge \min \{p, 1- p\}$, the result then follows.
\end{proof}

\begin{corollary} \label{col:day1-covar}
Assume the setting Lemma \ref{lem:day1-to-binom}.
Let $W_1, W_2, \ldots, W_r$ be arbitrary constants in $\{0, 1\}$, then
\begin{equation*}
\Cov {\I_1(u_1), \I_1(u_2) \mid W_{u_1v_i} = W_{u_2v_i} = W_i, \ i = 1,\ldots,r }  \le \frac{4C_0^2(1 - 2\sigma^2)^2}{n - r - 2}.
\end{equation*}
\end{corollary}

\begin{proof}
We impose the condition $\big\{W_{u_1v_i} = W_{u_2v_i} = W_i, \ i = 1,\ldots,r\big\}$ on every event within this proof so as not to repeat it in the equations. Note that the covariances are not automatically zero, as the indicators are not independent. By definition 
\begin{equation*}
\Cov{\I_1(u_1), \I_1(u_2)} \ = \ \Pr\left(u_1, u_2 \in R_1\right) - \Pr\left(u_1 \in R_1\right) \Pr\left(u_2 \in R_1\right).
\end{equation*}
Consider the event $\{u_1, u_2 \in R_1\}$,  $\Pr\left(u_1, u_2 \in R_1\right)$ can be written as 
$$  \Pr\left(u_1, u_2 \in R_1 \mid u_1 \adj u_2\right)\Pr(u_1 \adj u_2) + \ \Pr\left(u_1, u_2 \in R_1 \mid u_1 \not\adj u_2\right)\Pr(u_1 \not\adj u_2). $$
Notice  that after we specify the adjacency between $u_1$ and $u_2$, the neighborhoods of $u_1$ and $u_2$ are independent. Thus, we can write the above as
\begin{equation*}
\begin{aligned}
& p \ \cdot \ \underbrace{\Pr\left(u_1 \in R_1 \mid u_1 \adj u_2\right)}_{a_1} \cdot  \underbrace{\Pr\left(u_2 \in R_1 \mid u_1 \adj u_2\right)}_{a_2}
\\
& \ + \ (1-p) \ \cdot \ \underbrace{\Pr\left(u_1 \in R_1 \mid u_1 \not\adj u_2\right)}_{b_1} \cdot \underbrace{\Pr\left(u_2 \in R_1 \mid u_1 \not\adj u_2\right)}_{b_2}.
\end{aligned}
\end{equation*}
 Using shorthand  $q: = 1 - p$,  we obtain
 $\ \Pr\left(u_1, u_2 \in R_1\right) = pa_1a_2 + qb_1b_2.$
Consider the product $\Pr\left(u_1 \in R_1\right) \Pr\left(u_2 \in R_1\right)$. Splitting up the two events by $\{u_1\adj u_2 \}$ gives
$\Pr\left(u_1 \in R_1\right) =  pa_1 + qb_1$ and $\ \Pr\left(u_2 \in R_1\right) = pa_2 + qb_2.$
Putting everything together, we have
\begin{equation} \label{eq:nbh-temp34}
\begin{aligned}
\Cov{\I_1(u_1), \I_1(u_2)} \ & = \ pa_1a_2 + qb_2b_2 - \left(pa_1 + qb_1\right) \left(pa_2 + qb_2 \right) \\
\ & = \ pq(a_1 - b_1)(a_2 - b_2) = \sigma^2(a_1 - b_1)(a_2 - b_2).
\end{aligned}
\end{equation}
We next analyze the relationship between $a_1$ and $b_1$. (The analysis for $a_2$ and $b_2$ is similar.) Apply Lemma \ref{lem:day1-to-binom} to the set $S' = \{v_1, v_2, \ldots, v_r, u_2\}$, we get
\begin{equation*}
\begin{aligned}
a_1 & = \Pr \left(u\in R_1 \mid W_{u_1v_i} = W_i, W_{u_1u_2} = 1 \right)
= \Pr \left(Y_1 > Y_2 - T_1  - \J_0(u_2) \right), \\
b_1 & = \Pr \left(u\in R_1 \mid W_{u_1v_i} = W_i, W_{u_1u_2} = 0 \right)
= \Pr \left(Y_1 > Y_2 - T_1 \right),
\end{aligned}
\end{equation*}
where $ \ \left\{
\begin{array}{ll}
Y_1 & \sim \Bin\left(\card{R_0} - T_2, \ p \right) \\
Y_2 & \sim \Bin\left(\card{B_0} + T_2 - r - 2, \ p \right)
\end{array} \right.$, for $ \ \left\{
\begin{array}{ll}
T_1 & \defeq \ \sum_{i=1}^r \J_0(v_i)W_i - \I_0(u_1) \\
T_2 & \defeq \ \sum_{i=1}^r \I_0(v_i) + \I_0(u_2) + \I_0(u)
\end{array}\right.$.

\noindent Case analysis on $\J_0(u_2)$:
$ \left\{
\begin{array}{ll}
\J_0(u_2) = 1 & \!\!  \Rightarrow \ a_1 - b_1 = \Pr\left(Y_1 = Y_2 - T_1\right) \\
 \J_0(u_2) = -1 & \!\! \Rightarrow \ b_1 - a_1 = \Pr\left(Y_1 = Y_2 - T_1 + 1\right)
\end{array} \right.
$.

\noindent In any case, by Lemma \ref{thm:1/sqrt{n}} we have
\begin{equation*}
\card{a_1-b_1} \ \le \ \frac{2C_0\left(1 - 2\sigma^2\right)}{\sigma\sqrt{(\card{R_0} - T_2) + \card{B_0} + T_2 - r - 2}}
\ = \ \frac{2C_0\left(1 - 2\sigma^2\right)}{\sigma\sqrt{n - r - 2}}.
\end{equation*}

\noindent With  the same analysis  for $a_2$ and $b_2$, we obtain
$\card{a_2-b_2} \le \frac{2C_0\left(1 - 2\sigma^2\right)}{\sigma\sqrt{n-r-2}}$.
By Equation \eqref{eq:nbh-temp34},  we have 
\begin{equation*} \label{eq:nbh-covariance} 
\Cov{\I_1(u_1), \I_1(u_2)} \le \sigma^2\left(\frac{2C_0\left(1 - 2\sigma^2\right)}{\sigma\sqrt{n-r-2}}\right)^2 
= \frac{4C_0^2\left(1 - 2\sigma^2\right)^2}{n-r-2}.
\end{equation*}\\[-3em]
\end{proof}

Corollaries \ref{col:day1-exp} and \ref{col:day1-covar} are powerful in the sense they are useful in assessing the number of Red vertices in any subset of $V$, under broad conditions about their neighborhoods. Now we apply them to bound the number of Blue vertices in Day 2 with lemma in the next part.

\subsection{Proof of Theorem \ref{thm:const-advantage-day-1}}


Set $X \defeq \card{R_1} = \sum_{v\in V}\I_1(v)$. We aim to lower-bound the probability that $X < \frac{n + 1}{2} + d\sqrt{n}$ for some constant $d$.
For each $v\in V$, applying Corollary \ref{col:day1-exp} for the set $W = \varnothing$, we get
\begin{equation*}
\E{\I_1(v)} \ge \frac{1}{2}
+ \Phi_0\left(\frac{2pc + \min\left\{p, 1 - p\right\}}{\sigma\sqrt{n - 1}}\right)
- \frac{C_0\left(1 - 2\sigma^2\right)}{\sigma\sqrt{n - 1}}.
\end{equation*}
%
Therefore
\begin{equation} \label{eq:expectation-bound0}
\begin{aligned}
\E{X} & \ge n\left[\frac{1}{2}
+ \Phi_0\left(\frac{2pc + \min\left\{p, 1 - p\right\}}{\sigma\sqrt{n - 1}}\right)
- \frac{C_0\left(1 - 2\sigma^2\right)}{\sigma\sqrt{n - 1}}\right] \\
& \ge \frac{n}{2} + \left[\sqrt{n - 1}\Phi_0\left(\frac{2pc + \min\left\{p, 1 - p\right\}}{\sigma\sqrt{n - 1}}\right)
- \frac{C_0\left(1 - 2\sigma^2\right)}{\sigma} \right] \sqrt{n},
\end{aligned}
\end{equation}
By \eqref{eq:d-n-condition-defn}, we further have
\begin{equation*}
\begin{aligned}
& \E{X} - \left(\frac{n + 1}{2} + d\sqrt{n}\right) \\
& \ \ge  \ \left[\sqrt{n - 1} \Phi_0 \left( \frac{2pc + \min\{p, 1- p\}}{\sigma \sqrt{n - 1}} \right) - \frac{C_0(1 - 2\sigma^2)}{\sigma} - d - \frac{1}{2\sqrt{n}}\right] \sqrt{n} \ > \ 0.
\end{aligned}
\end{equation*}
Now consider $\Var{X} \displaystyle = \sum_{v\in V}\Var{\I_1(v)} + 2\sum_{v_1\neq v_2}\Cov{\I_1(v_1), \I_1(v_2)}$. Since $\I_1(v)$ is Bernoulli, $\Var{\I_1(v)} \le 1/4$. For each $v_1\neq v_2$, applying Corollary \ref{col:day1-covar} for $W = \varnothing$, we get
\begin{equation*}
\Cov{\I_1(v_1), \I_1(v_2)} \le \frac{4C_0^2\left(1 - 2\sigma^2\right)^2}{n-2}.
\end{equation*}
Therefore
\begin{equation} \label{eq:variance-bound}
\Var{X} \ \le \ \frac{n}{4} + 2\frac{4C_0^2\left(1 - 2\sigma^2\right)^2}{n-2} \binom{n}{2}
= \frac{n}{4} + 4C_0^2\left(1 - 2\sigma^2\right)^2 \frac{n(n-1)}{n - 2}.
\end{equation}
%
%
Now by Chebyshev' s inequality 
\begin{equation*} \label{eq:chebyshev-X-bound}
\begin{aligned}
& \Pr \left(\card{B_1} > \frac{n - 1}{2} - d\sqrt{n} \right) \ = \ \Pr \left(X < \frac{n + 1}{2} + d\sqrt{n} \right) \ \le \ \frac{\Var{X}}{\left(\E{X} - \frac{n+1}{2} - d\sqrt{n}\right)^2}  \\
&  \ \le \ \frac{\frac{n}{4} + 4C_0^2\left(1 - 2\sigma^2\right)^2\cdot \frac{n(n-1)}{n - 2}}
{\left[\sqrt{n - 1} \Phi_0 \left( \frac{2pc + \min\{p, 1- p\}}{\sigma \sqrt{n - 1}} \right) - \frac{C_0(1 - 2\sigma^2)}{\sigma} - d - \frac{1}{2\sqrt{n}}\right]^2 n}
\ = \ Q.
\end{aligned}
\end{equation*}
The proof for Theorem \ref{thm:const-advantage-day-1} is complete.

\section{DAY Two and after } \label{sec:universal-win}

Next, we analyze the situation after Day 1. We handle the loss of independence with a method called \emph{universal reduction}.

\begin{definition} \label{defn:uni-reduce}
	A graph $G = (V, E)$ is said to \emph{universally reduce} $m_1$ to $m_2$, where $m_1 \ge m_2$, if for any coloring of $V$ where the Blue set has at most $m_1$ vertices,  the Blue set in the next day has at most $m_2$ vertices.
	\end{definition}
	
We write  $G: m_1 \univto m_2 $.
Notice that it is irrelevant to specify the day in this definition. 
The following lemma is the simplest example of universal reduction, where a sufficiently small Blue camp can be reduced to zero when all vertices have more neighbors than twice its size. It is also useful as a final step for `` shrinking sequences'', i.e. arguments involving shrinking the Blue camp gradually to zero.

\begin{lemma} \label{lem:degree-bound}
For each $p, s\in (0, 1)$ and $n\in \N$,
\begin{equation*}
\Pr \left(G: \frac{1}{2}sp(n-1) \univto 0 \right) \ge 1 - n\exp\left(- (1 - s + s\log s)p(n-1) \right).
\end{equation*}
\end{lemma}
\begin{proof}
In a $G(n, p)$ graph, $d(v)$ is a sum of $(n-1)$ $\Bin(1, p)$ random variables, so Chernoff's inequality implies that for each $\lambda > 0$, we have
\begin{equation*}
\begin{aligned}
& \Pr \left(d(v) < sp(n-1) \right) \ \le \ \exp\left((n-1)\lambda sp + (n-1)\log\left(pe^{-\lambda} + 1 - p\right)\right) \\
& \le \ \exp\left((n-1)\lambda sp + (n-1)\left(pe^{-\lambda} - p\right)\right)
\ \le \ \exp \left( - (1 - e^{-\lambda} - \lambda s)p(n - 1) \right)
\end{aligned}
\end{equation*}
Letting $\lambda = - \log s$, the term $1 - e^{-\lambda} - \lambda s$ becomes $1 - s + s\log s$. By a union bound, the probability that all vertices have more than $sp(n-1)$ neighbors is at least $1 - n\exp\left(- (1 - s + s\log s)p(n-1) \right)$. Given this, a Blue camp of $sp(n-1)/2$ members surely vanishes the next day since it cannot form a majority in any vertex's neighborhood. The result then follows.
\end{proof}

In order to form longer shrinking sequences, we need more complicated universal reduction arguments. The next few results give the right tools to handle them.

\begin{definition}
Given a graph $G = (V, E)$ with a coloring, a subset $S$ of $V$ is called \emph{bad} if all its members turn Blue the next day.
\end{definition}

In other words, if the current day is $0$, a bad set is a subset of $B_1$. The following combinatorial lemma gives a formula helpful for determining when a set is bad.

\begin{lemma} \label{lem:bad-set-formula}
Given a graph $G = (V, E)$ with a coloring $(R, B)$. For each $v\in V$, let $~\dif(v) \defeq \card{\Gamma(v) \cap R} - \card{\Gamma(v) \cap B}$. For each $S\subseteq V$, let $~\dif(S) \defeq \sum_{v\in S}\dif(v)$. Then
\begin{equation*}
\dif(S) \ = \sum_{\{u, v\} \subset S\cap R}\left(2W_{uv}\right) - \sum_{\{u, v\} \subset S\cap B}\left(2 W_{uv}\right)
		+ \sum_{u\in S} \sum_{v\in R\setminus S}W_{uv} - \sum_{u\in S} \sum_{v\in R\setminus B}W_{uv} .
\end{equation*}
\end{lemma}

\begin{proof}
  We break down each $\dif(v)$ and $\dif(S)$ as follows:
	\begin{equation*}
	\begin{aligned}
	\dif(v) & \ = \ \sum_{u\in R\cap S}W_{vu} + \sum_{u\in R\setminus S}W_{vu} - \sum_{u\in B\cap S}W_{vu} - \sum_{u\in B\setminus S}W_{vu} \\
	\dif(S) & \ = \ \sum_{v\in S} \sum_{u\in R\cap S}W_{vu} + \sum_{v\in S} \sum_{u\in R\setminus S}W_{vu} - \sum_{v\in S} \sum_{u\in B\cap S}W_{vu} - \sum_{v\in S} \sum_{u\in B\setminus S}W_{vu}.
	\end{aligned}
	\end{equation*}
	We have
	\begin{equation*}
	\sum_{v\in S} \sum_{u\in R\cap S}W_{vu}
	= \sum_{v\in R\cap S} \sum_{u\in R\cap S}W_{vu} + \sum_{v\in B\cap S} \sum_{u\in R\cap S}W_{vu}
	= \sum_{\{u, v\} \subset S\cap R}\left(2W_{uv}\right) + \sum_{v\in B\cap S} \sum_{u\in R\cap S}W_{vu}.
	\end{equation*}
	Similarly, $~\displaystyle \sum_{v\in S} \sum_{u\in B\cap S}W_{vu}
	= \sum_{\{u, v\} \subset S\cap B}\left(2W_{uv}\right) + \sum_{v\in B\cap S} \sum_{u\in R\cap S}W_{vu}$.
	Therefore
	\begin{equation*}
	\sum_{v\in S} \sum_{u\in R\cap S}W_{vu} - \sum_{v\in S} \sum_{u\in B\cap S}W_{vu}
	\ = \ \sum_{\{u, v\} \subset S\cap R}\left(2W_{uv}\right) - \sum_{\{u, v\} \subset S\cap B}\left(2 W_{uv}\right)
	\end{equation*}
	The desired identity then follows.
\end{proof}

\begin{lemma} \label{lem:uni-bad-set-bound}
	Let $p\in (0, 1)$, $n, n_0 \in \mathbb{N}$, $n_0 < \frac{n}{2}$. Then for all $m\in \mathbb{N}, m\le n$,
	\begin{equation*}
	\Pr \left(G: n_0\univto m - 1\right) \ge 1 - \frac{4^n}{n} \exp\left(-\frac{2p^2(n - 2n_0 - 1)^2m}{n + m - 2} \right).
	\end{equation*}
\end{lemma}

\begin{proof}
Consider a subset $S$ of $V$ with $m$ elements. We will first bound the probability that $S$ is bad. Let $(R, B)$ be the initial coloring with $\card{B} = n_0 < n - n_0 = \card{R}$.
By Lemma \ref{lem:bad-set-formula}, we have
\begin{equation*}
\dif(S) \ = \sum_{\{u, v\} \subset S\cap R}\left(2W_{uv}\right) - \sum_{\{u, v\} \subset S\cap B}\left(2 W_{uv}\right)
		+ \sum_{u\in S} \sum_{v\in R\setminus S}W_{uv} - \sum_{u\in S} \sum_{v\in R\setminus B}W_{uv} .
\end{equation*}
	This sum consists of independent variables, so we can apply  Hoeffding concentration bound over $\dif(S)$. Firstly,
	\begin{equation}  \label{eq:temp97}
	\E{\dif(S)} = p\card{S}(\card{R} - \card{B}) - p(\abs{S\cap R} - \abs{S\cap B}) \ge pm(n - 2n_0 - 1).
	\end{equation}
	Moreover, each $W_{uv}$ takes values in $[0, 1]$ (a range of length 1) and $2W_{uv}$ takes values in $[0, 2]$ (a range of length 2), so the sum of squares of these lengths are
	\begin{equation*}
	\begin{aligned}
	F & \ = \ 4\binom{\card{S\cap R}}{2} + 4\binom{\card{S\cap B}}{2}
	+ \card{S}\card{R\setminus S} + \card{S}\card{B\setminus S} \\
	& \ = \ \card{S}(n - 2 + \card{S}) - 4\abs{S\cap R}\abs{S\cap B}
	\le m(n - 2 + m).
	\end{aligned}
	\end{equation*}
	By Hoeffding's inequality:
	\begin{equation*}
	\begin{aligned}
	\Pr\left(S\subseteq B_1\mid R_0, B_0\right)
	& \ \le \ \Pr\left(\dif(S) \le 0\right)
	\ = \ \Pr\left(\dif(S) - \E{\dif(S)}\le -\E{\dif(S)}\right) \\
	& \ \le \ \exp\left(-\frac{\E{\dif(S)}^2}{F}\right)
	\ \le \ \exp\left(-\frac{2p^2(n - 2n_0 - 1)^2m}{n - 2 + m}\right).
	\end{aligned}
	\end{equation*}
	Applying a double union bound over choices of $S$ and $(R, B)$, noting that there are $\binom{n}{n_0}\binom{n}{m} \le 4^n/n$ choices, we have
	\begin{equation*}
 \Pro \left(\exists (R, B). \exists S. \ \card{B} = n_0, \card{S} = m, S\subset B_1 \right)
 \ \le \ \frac{4^n}{n} \exp\left(-\frac{2p^2(n - 2n_0 - 1)^2m}{n + m - 2} \right).
	\end{equation*}
	Taking the complement event, we get the desired result.
\end{proof}

Now we use Lemma \ref{lem:uni-bad-set-bound} to prove that within two days, the Blue set can be shrunk from a size of $\frac{n}{2} - \text{O}(\sqrt{n})$ to size $\frac{n}{2} - \text{O}(n)$, then $\text{O}(pn)$, with high probability.

\begin{lemma}  \label{prop:uni-sqrt-advantage-ex}
	With $C_1 = \sqrt{3\log 2}$ defined in Section \ref{sec:introduction}, let $p\in (0, 1)$ and $n\in \mathbb{N}$. Then for all $\eps_1, \eps_2 > 0$, let $A(\eps_1, \eps_2) = (1 - 2\eps_1)p\eps_2(p\eps_2 + C_1) - \eps_1\log 2$, we have
\begin{equation} \label{eq:uni-sqrt-advantage-ex}
\Pr \left(G: \frac{n-1}{2} - \left(\frac{C_1}{2p} + \eps_2\right)\sqrt{n} \univto \left(\frac{1}{2} - \eps_1\right) n \right) \ \ge \ 1 - \frac{1}{n}\exp\left(-\frac{8n}{3}A \right).
\end{equation}
\end{lemma}

A routine calculation shows that if $\	\eps_1 < \left(2 + \frac{\log 2}{p\eps_2(p\eps_2 + C_1)}\right)^{-1}$,
then the RHS of \eqref{eq:uni-sqrt-advantage-ex} tends to $1$ as $n\to +\infty$.
\begin{proof}
	Let $n_2 \defeq \left\lfloor \frac{n-1}{2} - \left(\frac{C_1}{2p} + \eps_2\right)\sqrt{n} \right\rfloor$ and $m \defeq \left\lceil \left(\frac{1}{2} - \eps_1\right)n \right\rceil$. By Lemma \ref{lem:uni-bad-set-bound},
	
	\begin{equation*}
	\begin{aligned}
	& \Pro \left(G: \frac{n-1}{2} - \left(\frac{C_1}{2p} + \eps_2\right)\sqrt{n} \univto \left(\frac{1}{2} - \eps_1\right) n \right)
	\ \ge \ \P \left(G: n_2 \univto m - 1 \right) \\
  & \ \ge \ 1 - \frac{4^n}{n} \exp \left(- \frac{2p^2(n - 2n_2 - 1)^2m}{n + m -2} \right),
  \end{aligned}
	\end{equation*}
	
	Since $\frac{m}{n + m -2} \ge \frac{m}{n + m} \ge \frac{1 - 2\eps_1}{3 - 2\eps_1}$ and $n - 2n_2 - 1 \ge \left(\frac{C_1}{p} + 2\eps_2 \right)\sqrt{n}$, we can bound the RHS of the above as follows
	\begin{equation*}
	\begin{aligned}
	 & \frac{4^n}{n} \exp \left(- \frac{2p^2(n - 2n_2 - 1)^2m}{n + m -2} \right)
	\le \frac{4^n}{n} \exp \left(-2p^2\left(\frac{C_1}{p} + 2\eps_2\right)^2 \frac{1 - 2\eps_1}{3 - 2\eps_1}n \right) \\
	& = \ \exp \left( - \frac{2n}{3 - 2\eps_1} \left[ \left(\sqrt{3\log2} + 2p\eps_2\right)^2 (1 - 2\eps_1) - (3 - 2\eps_1)\log 2 \right] - \log n \right) \\
	& \le \ \exp \left( - \frac{2n}{3} \left[ 4p\eps_2(p\eps_2 + \sqrt{3\log 2})(1 - 2\eps_1) - 4\eps_1\log 2 \right] - \log n \right).
	\end{aligned}
	\end{equation*}
	Substituting $\sqrt{3\log 2}$ for $C_1$, we get the desired expression on the RHS.
\end{proof}

\begin{lemma} \label{prop:uni-linear-advantage}
	Let $p, r\in (0, 1)$, $\eps_1\in \left(0, \frac{1}{2}\right)$ and $n\in \mathbb{N}$. Then
  \begin{equation*} \label{eq:uni-linear-advantage}
  \Pr \left(G: \left(\frac{1}{2} - \eps_1\right)n \univto rp(n-1) \right) \ge \ 1 - \frac{1}{n}\exp\left(-\frac{2rp^3(2\eps_1 n - 1)^2}{1 + rp} + 2n\log 2\right).
  \end{equation*}
\end{lemma}

\begin{proof}
	Let $n_2 \defeq \left\lfloor \left(\frac{1}{2} - \eps_1\right)n \right\rfloor$ and $m \defeq \lceil rp(n-1) \rceil$. Since $G: n_2 \univto m - 1$ implies $G: \left(\frac{1}{2} - \eps_1\right)n \univto rp(n-1)$, Lemma \ref{lem:uni-bad-set-bound} implies
	
	\begin{equation*} \label{eq:temp95}
	\Pro \left(G: \left(\frac{1}{2} - \eps_1\right)n \univto rp(n-1) \right) \ge 1 - \frac{4^n}{n} \exp\left(-\frac{2p^2(n - 2n_2 - 1)^2m}{n + m - 2} \right).
	\end{equation*}
	
	Since $n_2\le \left(\frac{1}{2} - \eps_1\right)n$, $(n - 2n_2 - 1)^2 \ge (2\eps_1 n - 1)^2$.	
	Furthermore,  $m\ge rp(n-1)$ so  $\frac{m}{n + m - 2} \ge \frac{rp(n-1)}{n + rp(n-1) - 1} = \frac{rp}{1 + rp}$. Therefore
	\begin{equation*}
\Pro \left(G: \left(\frac{1}{2} - \eps_1\right)n \univto rp(n-1) \right) \ge 1 - \frac{4^n}{n} \exp\left(-2p^2\cdot \frac{rp(2\eps_1 n - 1)^2}{1 + rp} \right),
	\end{equation*}
	which is equivalent to the claimed bound. 
\end{proof}

\section {Proof of Main Theorem} \label{sec:proof-main}

We prove Theorem \ref{thm:const-advantage} by putting together Lemmas \ref{prop:uni-linear-advantage}, \ref{prop:uni-sqrt-advantage-ex} and Theorem \ref{thm:const-advantage-day-1}.
\begin{proof}[Proof of Theorem \ref{thm:const-advantage}]
Recall the terms $n_t^B$ for $0\le t\le 4$ and $P_t$ for $1\le t\le 4$.

\noindent Since $n_3^B = rp(n-1)$ and $n_4^B = 0$, Lemma \ref{lem:degree-bound} implies $\P (\card{B_4}\le n_4^B\mid \card{B_3}\le n_3^B) \ge 1 - P_4$. Similarly, by universal reduction, Lemma \ref{prop:uni-linear-advantage} implies  $\P (\card{B_3}\le n_3^B\mid \card{B_2}\le n_2^B) \ge 1 - P_3$ and Lemma \ref{prop:uni-sqrt-advantage-ex} implies $\P (\card{B_2}\le n_2^B\mid \card{B_1}\le n_1^B) \ge 1 - P_2$. Finally, Theorem \ref{thm:const-advantage} is equivalent to $\P (\card{B_1}\le n_1^B\mid \card{B_0}\le n_0^B) \ge 1 - P_1$.

Part 1 of Theorem \ref{thm:const-advantage} is complete. Now by a union bound, the probability that all of the conditional events above occur is at least $1 - (P_1 + P_2 + P_3 + P_4)$. In that case, the event $\{R_4 = V\mid \card{B_0} = n^B \}$ occurs, the proof is complete. 
\end{proof}

\vskip2mm {\it Acknowledgement.} We would like to thank Q. A. Do, A, Ferber, A. Deneanu, J. Fox, E. Mossel,  and X. Chen for inspiring discussions, and H. V. Le and T. Can for proofreading the paper.

\bibliographystyle{plain}
\bibliography{references}

\end{document}